\documentclass{article}%
\usepackage{graphicx}
\usepackage{amsmath}
\usepackage{amsfonts}
\usepackage{amssymb}
\usepackage{setspace}
\usepackage[hmargin=3.5cm,vmargin=3.5cm]{geometry}%
\providecommand{\U}[1]{\protect\rule{.1in}{.1in}}
\providecommand{\U}[1]{\protect\rule{.1in}{.1in}}
\providecommand{\U}[1]{\protect\rule{.1in}{.1in}}
\providecommand{\U}[1]{\protect\rule{.1in}{.1in}}
\providecommand{\U}[1]{\protect\rule{.1in}{.1in}}
\providecommand{\U}[1]{\protect\rule{.1in}{.1in}}
\providecommand{\U}[1]{\protect\rule{.1in}{.1in}}
\providecommand{\U}[1]{\protect\rule{.1in}{.1in}}
\providecommand{\U}[1]{\protect\rule{.1in}{.1in}}
\providecommand{\U}[1]{\protect\rule{.1in}{.1in}}
\providecommand{\U}[1]{\protect\rule{.1in}{.1in}}
\providecommand{\U}[1]{\protect\rule{.1in}{.1in}}
\providecommand{\U}[1]{\protect\rule{.1in}{.1in}}
\providecommand{\U}[1]{\protect\rule{.1in}{.1in}}
\providecommand{\U}[1]{\protect\rule{.1in}{.1in}}
\providecommand{\U}[1]{\protect\rule{.1in}{.1in}}
\providecommand{\U}[1]{\protect\rule{.1in}{.1in}}
\providecommand{\U}[1]{\protect\rule{.1in}{.1in}}
\providecommand{\U}[1]{\protect\rule{.1in}{.1in}}
\providecommand{\U}[1]{\protect\rule{.1in}{.1in}}
\providecommand{\U}[1]{\protect\rule{.1in}{.1in}}
\providecommand{\U}[1]{\protect\rule{.1in}{.1in}}
\providecommand{\U}[1]{\protect\rule{.1in}{.1in}}
\providecommand{\U}[1]{\protect\rule{.1in}{.1in}}
\providecommand{\U}[1]{\protect\rule{.1in}{.1in}}
\providecommand{\U}[1]{\protect\rule{.1in}{.1in}}
\providecommand{\U}[1]{\protect\rule{.1in}{.1in}}
\providecommand{\U}[1]{\protect\rule{.1in}{.1in}}

\newtheorem{theorem}{Theorem}
{}

\newtheorem{lemma}{Lemma}
{}

\newtheorem{proposition}{Proposition}

\newenvironment{proof}[1][Proof]{\textbf{#1.} }{\ \rule{0.5em}{0.5em}}

\begin{document}

\title{On the Basis Property of the Root Functions of Sturm-Liouville Operators with
General Regular Boundary Conditions.}
\author{Cemile Nur and O. A. Veliev\\{\small Department of Mathematics, Dogus University, Kadik\"{o}y, \ Istanbul,
Turkey.}\\\ {\small E-mail: cnur@dogus.edu.tr;} {\small oveliev@dogus.edu.tr}}
\date{}
\maketitle

\begin{abstract}
We obtain the asymptotic formulas for the eigenvalues and eigenfunctions of
the Sturm-Liouville operators with general regular boundary conditions. Using
these formulas, we find sufficient conditions on the potential $q$ such that
the root functions of these operators do not form a Riesz basis.



\end{abstract}

\section{Introduction and Preliminary Facts}

In this paper we consider\ the operators generated in $L_{2}[0,1]$ by the
differential expression
\begin{equation}
l\left(  y\right)  =-y^{\prime\prime}+q(x)y
\end{equation}
and regular boundary conditions that are not strongly regular. Note that, if
the boundary conditions are strongly regular, then the root functions
(eigenfunctions and associated functions) form a Riesz basis (this result was
proved independently in [6], [9] and [17]). In the case when an operator is
regular but not strongly regular, the root functions generally do not form
even usual basis. However, Shkalikov [20, 21] proved that they can be combined
in pairs, so that the corresponding 2-dimensional subspaces form a Riesz basis
of subspaces.

To describe the results of this paper and preliminary results let us classify
all regular boundary conditions that are not strongly regular. One can readily
see from pages 62-63 of [18] that all regular boundary conditions that are not
strongly regular can be written in the form
\begin{align}
a_{1}y_{0}^{\prime}+b_{1}y_{1}^{\prime}+a_{0}y_{0}+b_{0}y_{1}  &
=0,\nonumber\\
c_{0}y_{0}+d_{0}y_{1}  &  =0,
\end{align}
if
\begin{equation}
b_{1}c_{0}+a_{1}d_{0}\neq0.
\end{equation}
and $\theta_{0}^{2}-4\theta_{1}\theta_{-1}=0,$ where , $a_{i},b_{i}%
,c_{0},d_{0}$, $i=0,1$, are complex numbers and $\theta_{0},\theta_{1}$ and
$\theta_{-1}$ are defined \ by
\begin{equation}
\frac{\theta_{-1}}{s}+\theta_{0}+\theta_{1}s=w_{1}\left(  b_{1}c_{0}%
+a_{1}d_{0}\right)  \left(  s+\frac{1}{s}\right)  +2\left(  a_{1}c_{0}%
+b_{1}d_{0}\right)  w_{1}%
\end{equation}
\ in p.63 of [18]. Thus, by (4), $\theta_{-1}=\theta_{1}=w_{1}\left(
b_{1}c_{0}+a_{1}d_{0}\right)  ,$ $\theta_{0}=2\left(  a_{1}c_{0}+b_{1}%
d_{0}\right)  w_{1},$ and hence the equality $\theta_{0}^{2}-4\theta_{1}%
\theta_{-1}=0$ implies that
\[
4\omega_{1}^{2}\left[  \left(  a_{1}c_{0}+b_{1}d_{0}\right)  ^{2}-\left(
b_{1}c_{0}+a_{1}d_{0}\right)  ^{2}\right]  =0,
\]
that is, $\left(  a_{1}^{2}-b_{1}^{2}\right)  \left(  c_{0}^{2}-d_{0}%
^{2}\right)  =0$ which means that at least one of the following conditions
holds:
\[
a_{1}=\pm b_{1},\text{ }c_{0}=\pm d_{0}.
\]

First suppose that $a_{1}=\left(  -1\right)  ^{\sigma}b_{1},$ where
$\sigma=0,1.$ This with (3) implies that both $a_{1}$ and $b_{1}$ are not zero
and at least one of $c_{0}$ and $d_{0}$ is not zero. If $c_{0}\neq0$, then (2)
can be written in the form
\begin{align}
y_{0}^{\prime}+\left(  -1\right)  ^{\sigma}y_{1}^{\prime}+\alpha_{1}y_{1}  &
=0,\nonumber\\
y_{0}+\alpha_{2}y_{1}  &  =0,
\end{align}
where $\alpha_{1}=\dfrac{b_{0}}{a_{1}}-\dfrac{a_{0}d_{0}}{a_{1}c_{0}}$,
$\alpha_{2}=\dfrac{d_{0}}{c_{0}}$, $a_{1},c_{0}\neq0$ and $\alpha_{2}%
\neq-\left(  -1\right)  ^{\sigma}$ due to (3).

Similarly, if $d_{0}\neq0$, then (2) can be transformed to
\begin{align}
y_{0}^{\prime}+\left(  -1\right)  ^{\sigma}y_{1}^{\prime}+\alpha_{3}y_{0}  &
=0,\nonumber\\
\alpha_{4}y_{0}+y_{1}  &  =0,
\end{align}
where $\alpha_{3}=\dfrac{a_{0}}{a_{1}}-\dfrac{b_{0}c_{0}}{a_{1}d_{0}}$,
$\alpha_{4}=\dfrac{c_{0}}{d_{0}}$, $a_{1},d_{0}\neq0$ and by (3) $\alpha
_{4}\neq-\left(  -1\right)  ^{\sigma}.$

Now suppose that $d_{0}=\left(  -1\right)  ^{\sigma}c_{0}.$ Arguing as in the
reductions of (5) and (6) we arrive at the boundary conditions%
\begin{align}
y_{0}^{\prime}+\beta_{1}y_{1}^{\prime}+\beta_{2}y_{1}  &  =0,\nonumber\\
y_{0}+\left(  -1\right)  ^{\sigma}y_{1}  &  =0,
\end{align}
where $\beta_{1}=\dfrac{b_{1}}{a_{1}},$ $\beta_{2}=\left(  \dfrac{b_{0}}%
{a_{1}}\mp\dfrac{a_{0}}{a_{1}}\right)  ,$ $a_{1},c_{0}\neq0$ and%
\begin{equation}
\beta_{1}\neq-\left(  -1\right)  ^{\sigma}%
\end{equation}
and the boundary conditions
\begin{align}
\beta_{3}y_{0}^{\prime}+y_{1}^{\prime}+\beta_{4}y_{1}  &  =0,\nonumber\\
y_{0}+\left(  -1\right)  ^{\sigma}y_{1}  &  =0,
\end{align}
where $\beta_{3}=\dfrac{a_{1}}{b_{1}},$ $\beta_{4}=\dfrac{b_{0}}{b_{1}}%
\mp\dfrac{a_{0}}{b_{1}},$ $b_{1},c_{0}\neq0$ and
\begin{equation}
\beta_{3}\neq-\left(  -1\right)  ^{\sigma}%
\end{equation}
for $\sigma=0,1.$

One can verify in the standard way that, the boundary conditions (5) and (6),
are the adjoint boundary conditions to (9) and (7), respectively, where
$\alpha_{3}=-\left(  -1\right)  ^{\sigma}\overline{\beta_{2}}$, $\alpha
_{4}=\overline{\beta_{1}}$ and $\alpha_{1}=\left(  -1\right)  ^{\sigma
}\overline{\beta_{4}}$, $\alpha_{2}=\overline{\beta_{3}}$.

Thus to consider all regular boundary conditions that are not strongly regular
it is enough to investigate the boundary conditions (7) and (9). Note that
these boundary conditions depend on two parameters. Let us describe the
special cases that were investigated.

Case $\left(  a\right)  $ The cases $\beta_{2},\beta_{4}=0$ , $\beta_{1}%
,\beta_{3}=\left(  -1\right)  ^{\sigma}$ in (7), (9) for $\sigma=1$ and
$\sigma=0$ coincide with the periodic and antiperiodic boundary
conditions\ respectively. These boundary conditions are the ones more commonly
studied. Therefore, let us briefly describe some historical developments
related to the Riesz basis property of the root functions of the periodic and
antiperiodic boundary value problems. First results were obtained by Kerimov
and Mamedov [8]. They established that, if
\[
q\in C^{4}[0,1],\ q(1)\neq q(0),
\]
then the root functions of the operator $L(q)$ form a Riesz basis in
$L_{2}[0,1],$ where $L(q)$ denotes the operator generated by (1) and the
periodic boundary conditions.

The first result in terms of the Fourier coefficients of the potential $q$ was
obtained by Dernek and Veliev [1]. Makin [11] extended this result for the
larger class of functions. Shkalilov and Veliev obtained in [22] more general
results which cover all results about periodic and antiperiodic boundary
conditions discussed above.

The other interesting results about periodic and antiperiodic boundary
conditions were obtained in [2-5, 7, 14-16, 23-25].

Case $\left(  b\right)  $ The cases $\beta_{2},\beta_{4}\neq0$ and $\beta
_{1},\beta_{3}=\left(  -1\right)  ^{\sigma}$ are investigated in [12, 13] and
it was proved that the system of the root functions of the Sturm-Liouville
operator corresponding to this case is a Riesz basis in $L_{2}\left(
0,1\right)  $ (see Theorem 1 of [12,13]).

Case $\left(  c\right)  $ The cases $\beta_{2},\beta_{4}=0$ and $\beta
_{1},\beta_{3}\neq\left(  -1\right)  ^{\sigma}$ are investigated in [12, 13]
and [19].

We call the boundary conditions (7) and (9) for $\beta_{2},\beta_{4}\neq0$ and
$\beta_{1},\beta_{3}\neq\left(  -1\right)  ^{\sigma}$ which are different from
the special cases $\left(  a\right)  ,$ $\left(  b\right)  $ and $\left(
c\right)  $ as general regular boundary conditions that are not strongly
regular. Note that in any case $\beta_{1},\beta_{3}\neq-\left(  -1\right)
^{\sigma}$ by (8) and (10). For the case $\left(  c\right)  $ and general
boundary conditions Makin [12, 13] proved that the systems of the root
functions of the Sturm-Liouville operators corresponding to these cases are
Riesz bases in $L_{2}\left(  0,1\right)  $ if and only if all large
eigenvalues are multiple. Note that this result is not effective, since the
conditions are given in implicit form and can not be verified for concrete
potentials. In [19] we find explicit conditions on potential such that the
system of the root functions of the Sturm-Liouville operator corresponding to
the case $\left(  c\right)  $ does not form a Reisz basis. Namely we proved
that if
\begin{equation}
\lim_{n\rightarrow\infty}\dfrac{\ln\left\vert n\right\vert }{ns_{2n}}=0,
\end{equation}
where $s_{n}=(q,\sin2\pi nt)$ and $\left(  .,.\right)  $ is the inner product
in $L_{2}\left[  0,1\right]  ,$ then the large eigenvalues of each of the
operators corresponding to these cases are simple for $\sigma=1$. Moreover, if
there exists a sequence $\left\{  n_{k}\right\}  $ such that (11) holds when
$n$ is replaced by $n_{k},$ then the root functions of these operators do not
form a Riesz basis. Similarly, if the condition
\[
\lim_{n\rightarrow\infty}\dfrac{\ln\left\vert n\right\vert }{ns_{2n+1}}=0
\]
holds instead of (11), then the same statements continue to hold for
$\sigma=0$.

In this paper we find explicit conditions on potential $q$ such that the
system of the root functions of the Sturm-Liouville operator generated by (1)
and the general regular boundary conditions does not form a Reisz basis.

Now let us describe briefly the main results of this paper. Let $T_{1}%
^{\sigma}(q)$ and $T_{2}^{\sigma}(q)$\ be the Sturm-Liouville operators
associated by the boundary conditions (7) and (9) respectively. Without loss
of generality we assume that
\[
\int_{0}^{1}q\left(  t\right)  dt=0.
\]
First we prove that if $q\in L_{1}\left[  0,1\right]  $ and
\begin{equation}
\int_{0}^{1}\sin\left(  2\pi nt\right)  q\left(  t\right)  dt=o\left(
\frac{1}{n}\right)
\end{equation}
then the large eigenvalues of $T_{1}^{\sigma}(q)$ and $T_{2}^{\sigma}(q)$\ for
$\sigma=1,$\ are simple. Moreover if there exists a subsequence $\left\{
n_{k}\right\}  $ such that (12) holds whenever $n$ is replaced by $n_{k}$,
then the system of the root functions of each operators $T_{1}^{\sigma}(q)$
and $T_{2}^{\sigma}(q)$\ for $\sigma=1,$\ does not form a Riesz basis. The
same results continue to hold for $T_{1}^{\sigma}(q)$ and $T_{2}^{\sigma}%
(q)$\ for $\sigma=0,$ if instead of (12) the condition
\begin{equation}
\int_{0}^{1}\sin(\left(  2n+1)\pi t\right)  q\left(  t\right)  dt=o\left(
\frac{1}{n}\right)  \tag{12a}%
\end{equation}
holds.

The other main result is the following: If the potential $q$ is an absolutely
continuous function and
\begin{equation}
q\left(  0\right)  +\left(  -1\right)  ^{\sigma}q\left(  1\right)  \neq
\frac{2\beta_{2}^{2}}{1-\beta_{1}^{2}}%
\end{equation}
then the large eigenvalues of $T_{1}^{\sigma}(q)$ for $\sigma=0,1$\ are simple
and the system of the root functions of $T_{1}^{\sigma}(q)$\ does not form a
Riesz basis. Similarly, if the condition
\begin{equation}
q\left(  0\right)  +\left(  -1\right)  ^{\sigma}q\left(  1\right)  \neq
\frac{2\beta_{4}^{2}}{\beta_{3}^{2}-1}%
\end{equation}
holds instead of (13), then the same results remain valid for $T_{2}^{\sigma
}(q)$ for $\sigma=0,1$. Moreover we obtain subtle asymptotic formulas for the
eigenvalues and eigenfunctions for the operators $T_{1}^{\sigma}(q)$ and
$T_{2}^{\sigma}(q)$ for both cases $q\in L_{1}\left[  0,1\right]  $ and $q$ is
an absolutely continuous function.

Note that the general cases we investigate in this paper are essentially
different from the case $\left(  c\right)  $ as the method of investigations
and obtained results.

\section{Main Results}

We will focus only on the operator $T_{1}^{1}\left(  q\right)  $. The
investigations of the operators $T_{1}^{0}\left(  q\right)  ,T_{2}^{0}\left(
q\right)  $ and $T_{2}^{1}\left(  q\right)  $ are similar. First let us prove
the following simple proposition about $T_{1}^{1}\left(  0\right)  $. Note
that the simplest case $q(x)\equiv0$ was completely solved in [10]. Here we
write the asymptotic formulas for the eigenvalues of $T_{1}^{1}\left(
0\right)  $ in the form we need.

\begin{proposition}
The square roots (with nonnegative real part) of the eigenvalues of the
operator $T_{1}^{1}\left(  0\right)  $ consist of the sequences $\left\{
\mu_{n,1}\left(  0\right)  \right\}  $ and $\left\{  \mu_{n,2}\left(
0\right)  \right\}  $ satisfying
\begin{equation}
\mu_{n,1}\left(  0\right)  =2\pi n,
\end{equation}%
\begin{equation}
\mu_{n.2}\left(  0\right)  =2\pi n+\frac{\beta_{2}}{\beta_{1}-1}\frac{1}{\pi
n}+O\left(  \frac{1}{n^{2}}\right)  .
\end{equation}

\end{proposition}

\begin{proof}
Using the fundamental solutions $e^{i\mu x}$ and $e^{-i\mu x}$ of
$-y^{\prime\prime}=\lambda y$ where $\mu=\sqrt{\lambda}$, one can readily see
that the characteristic determinant $\Delta_{0}\left(  \mu\right)  $ of
$T_{1}^{1}\left(  0\right)  $ has the form
\begin{equation}
\Delta_{0}\left(  \mu\right)  =\left(  1-e^{i\mu}\right)  \left(  i\mu
+\beta_{1}i\mu e^{-i\mu}-\beta_{2}e^{-i\mu}\right)  +\left(  i\mu+\beta
_{1}i\mu e^{i\mu}+\beta_{2}e^{i\mu}\right)  \left(  1-e^{-i\mu}\right)
=0.\nonumber
\end{equation}
After simplifying this equation, we have%
\begin{equation}
\Delta_{0}\left(  \mu\right)  =\left(  1-e^{-i\mu}\right)  \left[  i\mu\left(
\beta_{1}-1\right)  \left(  e^{i\mu}-1\right)  +\beta_{2}\left(  e^{i\mu
}+1\right)  \right]  =0
\end{equation}
which is equivalent to
\begin{equation}
1-e^{-i\mu}=0\text{ or }f(\mu)=0
\end{equation}
where
\begin{equation}
f(\mu)=e^{i\mu}-1-\frac{i\beta_{2}}{\beta_{1}-1}\frac{e^{i\mu}+1}{\mu}%
=e^{i\mu}-1+O\left(  \frac{1}{\mu}\right)
\end{equation}
The solution of the first equation in (18) is $\mu_{n,1}\left(  0\right)
=2\pi n$ for $n\in\mathbb{Z}$, that is, (15) is proved.

To prove (16), we estimate the roots of (19). Using Rouche's theorem on the
circle $\left\{  \mu:\text{ }\left\vert \mu-2\pi n\right\vert =\dfrac{c}%
{n}\right\}  $ for some constant $c$, one can easily see that, the roots of
(19) has the form%
\begin{equation}
\mu_{2,n}^{0}=2\pi n+\xi\text{ }\&\text{ }\xi=O\left(  \dfrac{1}{n}\right)  .
\end{equation}
Now we prove that
\begin{equation}
\xi=\frac{\beta_{2}}{\beta_{1}-1}\frac{1}{\pi n}+O\left(  \frac{1}{n^{2}%
}\right)  .
\end{equation}
For this, let us consider the roots of (19) in detail. By (20) and (19) we
have
\begin{equation}
e^{i\left(  2\pi n+\xi\right)  }-1=\frac{i\beta_{2}}{\beta_{1}-1}%
\frac{2+O\left(  \frac{1}{n}\right)  }{2\pi n+O\left(  \frac{1}{n}\right)
}=\frac{2i\beta_{2}}{\beta_{1}-1}\frac{1}{2\pi n}+O\left(  \frac{1}{n^{2}%
}\right)  .
\end{equation}
On the other hand, using Maclaurin expansion of $e^{i\xi}$ and taking into
account the second equality of (20) we see that
\[
e^{i\left(  2\pi n+\xi\right)  }-1=i\xi+O\left(  \frac{1}{n^{2}}\right)
\]
This with (22) gives us (21). Now (16) follows from (20) and (21). Lemma is proved.
\end{proof}

For $q\neq0$ it is known that (see (21) of [13]) the characteristic polynomial
of $T_{1}^{1}\left(  q\right)  $ has the form
\begin{equation}
\Delta\left(  \mu\right)  =\Delta_{0}\left(  \mu\right)  -\frac{\beta_{1}%
+1}{2}\left\{  e^{i\mu}\left(  c_{\mu}-is_{\mu}\right)  -e^{-i\mu}\left(
c_{\mu}+is_{\mu}\right)  \right\}  +o\left(  \frac{1}{\mu}\right)  ,
\end{equation}
where $\Delta_{0}\left(  \mu\right)  $\ is defined in (17) and
\begin{equation}
c_{\mu}=\int_{0}^{1}\cos\left(  2\mu t\right)  q\left(  t\right)  dt,\text{
}s_{\mu}=\int_{0}^{1}\sin\left(  2\mu t\right)  q\left(  t\right)  dt.
\end{equation}
After some arrangements (23) can be written in the form
\begin{equation}
\Delta\left(  \mu\right)  =\Delta_{0}\left(  \mu\right)  -\frac{\beta_{1}%
+1}{2}e^{-i\mu}\left\{  c_{\mu}\left(  e^{2i\mu}-1\right)  -is_{\mu}\left(
e^{2i\mu}+1\right)  \right\}  +o\left(  \frac{1}{\mu}\right)  .
\end{equation}
Using (17) in this formula we obtain%
\begin{gather*}
\Delta\left(  \mu\right)  =\left(  1-e^{-i\mu}\right)  \left[  i\mu\left(
\beta_{1}-1\right)  \left(  e^{i\mu}-1\right)  +\beta_{2}\left(  e^{i\mu
}+1\right)  \right]  -\\
-\frac{\beta_{1}+1}{2}e^{-i\mu}\left\{  c_{\mu}\left(  e^{2i\mu}-1\right)
-is_{\mu}\left(  e^{2i\mu}+1\right)  \right\}  +o\left(  \frac{1}{\mu}\right)
\\
=\left(  1-e^{-i\mu}\right)  \left[  i\mu\left(  \beta_{1}-1\right)  \left(
e^{i\mu}-1\right)  +\beta_{2}\left(  e^{i\mu}+1\right)  -\frac{\beta_{1}+1}%
{2}c_{\mu}\left(  e^{i\mu}+1\right)  \right]  +\\
+i\left(  \beta_{1}+1\right)  s_{\mu}\cos\mu+o\left(  \frac{1}{\mu}\right)  .
\end{gather*}

Therefore the characteristic determinant $\Delta\left(  \mu\right)  $, can be
written as
\begin{equation}
\Delta\left(  \mu\right)  =\Delta_{1}\left(  \mu\right)  +i\left(  \beta
_{1}+1\right)  s_{\mu}\cos\mu+o\left(  \frac{1}{\mu}\right)  .
\end{equation}
where
\begin{equation}
\Delta_{1}\left(  \mu\right)  =\left(  1-e^{-i\mu}\right)  \left[  i\mu\left(
\beta_{1}-1\right)  \left(  e^{i\mu}-1\right)  +\left(  \beta_{2}-\frac
{\beta_{1}+1}{2}c_{\mu}\right)  \left(  e^{i\mu}+1\right)  \right]  .
\end{equation}
To obtain the asymptotic formulas for the eigenvalues of $T_{1}^{1}\left(
q\right)  $ first let us consider the roots of $\Delta_{1}\left(  \mu\right)
.$

\begin{lemma}
The roots of the function $\Delta_{1}\left(  \mu\right)  $ consist of the
sequences $\left\{  \mu_{n,1}^{1}\right\}  $ and $\left\{  \mu_{n,2}%
^{1}\right\}  $ such that
\begin{equation}
\mu_{n,1}^{1}=2\pi n,\text{ }n\in\mathbb{Z},
\end{equation}%
\begin{equation}
\mu_{n.2}^{1}=2\pi n+\frac{\beta_{2}}{\beta_{1}-1}\frac{1}{\pi n}+o\left(
\frac{1}{n}\right)  .
\end{equation}

\end{lemma}

\begin{proof}
The zeros of $\Delta_{1}\left(  \mu\right)  $ are the zeros of the equations
\[
1-e^{-i\mu}=0\text{,}%
\]
and%
\[
g\left(  \mu\right)  =:e^{i\mu}-1+\frac{1}{\beta_{1}-1}\left(  \beta_{2}%
-\frac{\beta_{1}+1}{2}c_{\mu}\right)  \frac{e^{i\mu}+1}{i\mu}=0.
\]
The roots of the first equation are $2\pi n$ for $n\in\mathbb{Z}$, that is
(28) holds. By definition of $f\left(  \mu\right)  $ (see (19)) we have
\[
g\left(  \mu\right)  =f(\mu)-\frac{\frac{\beta_{1}+1}{2}c_{\mu}}{\beta_{1}%
-1}\frac{e^{i\mu}+1}{i\mu}.
\]
Since $c_{\mu}=o(1)$, there exists a sequence $\delta_{n}$ such that
$\delta_{n}=o(1)$\ and
\begin{equation}
\left\vert g\left(  \mu\right)  -f(\mu)\right\vert <\frac{\delta_{n}}{n}%
\end{equation}
for $\mu\in U(2\pi n),$ where $U(2\pi n)$ is $O\left(  \dfrac{1}{n}\right)
$-neighborhood of $2\pi n$.

Now to estimate the zeros of $g\left(  \mu\right)  $, we use Rouche's theorem
for the functions $f\left(  \mu\right)  $ and $g\left(  \mu\right)  $ on the
circle
\begin{equation}
\gamma_{n}=\left\{  \mu:\mid\mu-\mu_{n,2}\left(  0\right)  \mid=\frac
{\varepsilon_{n}}{n}\right\}  ,
\end{equation}
where $\mu_{n,2}\left(  0\right)  $ is defined in (16) and $\varepsilon_{n}$
is chosen so that
\begin{equation}
\varepsilon_{n}=o(1)\ \&\text{ }\delta_{n}=o(\varepsilon_{n}).
\end{equation}
For this let us estimate $\left\vert f\left(  \mu\right)  \right\vert $ on
$\gamma_{n}$ by using the Taylor series of $f(\mu)$\ about $\mu_{n,2}\left(
0\right)  :$
\[
f\left(  \mu\right)  =f^{\prime}\left(  \mu_{n,2}\right)  \left(  \mu
-\mu_{n,2}\right)  +\frac{f^{\prime\prime}\left(  \mu_{n,2}\right)  }%
{2!}\left(  \mu-\mu_{n,2}\right)  ^{2}+\cdots
\]
Since
\[
f^{\prime}(\mu)=ie^{i\mu}-\frac{i\beta_{2}}{\beta_{1}-1}\frac{ie^{i\mu}}{i\mu
}+O\left(  \dfrac{1}{n^{2}}\right)  \sim1\text{, }f^{\prime\prime}(\mu
)\sim1,\ldots,
\]
there exist a constant $c>0$ such that $\left\vert f^{\prime}\left(
\mu\right)  \right\vert >c$ and
\begin{equation}
\left\vert f\left(  \mu\right)  \right\vert >c\frac{\varepsilon_{n}}{2n}%
\end{equation}
for $\mu\in\gamma_{n}$. Thus by (30)-(33) and Rouche's theorem, there exists a
root $\mu_{n,2}^{1}$ of $g\left(  \mu\right)  $ inside the circle (31).
Therefore (29), follows from (16).
\end{proof}

Now using (26), (27) and Lemma 1, we get one of the main results of this paper.

\begin{theorem}
$\left(  a\right)  $ If (12) holds, then the large eigenvalues of $T_{1}%
^{1}(q)$ are simple and the square roots (with nonnegative real part) of these
eigenvalues consist of two sequences $\{\mu_{n,1}\left(  q\right)  \}$ and
$\{\mu_{n,2}\left(  q\right)  \}$ satisfying the asymptotic formulas
\begin{equation}
\mu_{n,1}\left(  q\right)  =2\pi n+o\left(  \frac{1}{n}\right)  \text{, }%
\end{equation}%
\begin{equation}
\mu_{n,2}\left(  q\right)  =2\pi n+\frac{\beta_{2}}{\beta_{1}-1}\frac{1}{\pi
n}+o\left(  \frac{1}{n}\right)  .
\end{equation}
Moreover the normalized eigenfunctions $\varphi_{n,1}\left(  x\right)  $ and
$\varphi_{n,2}\left(  x\right)  $ corresponding to the eigenvalues $\left(
\mu_{n,1}\left(  q\right)  \right)  ^{2}$ and $\left(  \mu_{n,2}\left(
q\right)  \right)  ^{2}$satisfy the same asymptotic formula
\begin{equation}
\varphi_{n,j}\left(  x\right)  =\sqrt{2}\cos2\pi nx+O\left(  \frac{1}%
{n}\right)
\end{equation}
for $j=1,2$

$\left(  b\right)  $ If there exists a subsequence $\left\{  n_{k}\right\}  $
such that (12) holds whenever $n$ is replaced by $n_{k}$, then the system of
the root functions of $T_{1}^{1}(q)$ does not form a Riesz basis.
\end{theorem}

\begin{proof}
$\left(  a\right)  $ To prove (34) and (35), we show that the large roots of
$\Delta\left(  \mu\right)  $ lies in $o\left(  \frac{1}{n}\right)
$-neighborhood of the roots of $\Delta_{1}\left(  \mu\right)  $ by using
Rouche's theorem for $\Delta\left(  \mu\right)  $ and $\Delta_{1}\left(
\mu\right)  $\ on $\Gamma_{1}\left(  r_{n}\right)  $, $\Gamma_{2}\left(
r_{n}\right)  ,$ where
\begin{equation}
\Gamma_{j}\left(  r_{n}\right)  =\left\{  \mu:\text{ }\left\vert \mu-\mu
_{n,j}^{1}\right\vert =r_{n}\right\}  ,r_{n}=o\left(  \frac{1}{n}\right)
\end{equation}
and $\mu_{n,j}^{1}$ for $j=1,2$ are the roots of $\Delta_{1}\left(
\mu\right)  $. If $\mu\in\Gamma_{j}\left(  r_{n}\right)  $ for $j=1,2$ then by
(12) $s_{\mu}=o\left(  \dfrac{1}{n}\right)  $ and by (26)
\begin{equation}
a\left(  \mu\right)  =:\left\vert \Delta\left(  \mu\right)  -\Delta_{1}\left(
\mu\right)  \right\vert <b_{n},\text{ }b_{n}=o\left(  \frac{1}{n}\right)  .
\end{equation}
We can choose $r_{n}$ so that
\begin{equation}
b_{n}=o\left(  r_{n}\right)  .
\end{equation}
Now let us estimate $\Delta_{1}\left(  \mu\right)  $ on the circles
$\Gamma_{1}\left(  r_{n}\right)  $, $\Gamma_{2}\left(  r_{n}\right)  $. By
(27)
\begin{equation}
\Delta_{1}\left(  \mu\right)  =\left(  1-e^{-i\mu}\right)  i\mu h\left(
\mu\right)
\end{equation}
where
\begin{equation}
h\left(  \mu\right)  =\left(  \beta_{1}-1\right)  \left(  e^{i\mu}-1\right)
+\left(  \beta_{2}-\frac{\beta_{1}+1}{2}c_{\mu}\right)  \frac{e^{i\mu}+1}%
{i\mu}.
\end{equation}
It follows from (28), (29) and (37) that if $\mu\in\Gamma_{1}\left(
r_{n}\right)  $ and $\mu\in\Gamma_{2}\left(  r_{n}\right)  $ then $\mu=2\pi
n+r_{n}e^{i\theta}$ and $\mu=2\pi n+\dfrac{\beta_{2}}{\beta_{1}-1}\dfrac
{1}{\pi n}+r_{n}e^{i\theta}+o\left(  \dfrac{1}{n}\right)  $ respectively,
where $\theta\in\left(  0,2\pi\right)  $. Therefore
\begin{equation}
\left(  1-e^{-i\mu}\right)  \sim r_{n}\text{,}%
\end{equation}
and%
\begin{equation}
\left(  1-e^{-i\mu}\right)  \sim\frac{1}{n}\text{,}%
\end{equation}
on$\ \Gamma_{1}\left(  r_{n}\right)  $ and $\Gamma_{2}\left(  r_{n}\right)  $
respectively, where $a_{n}\sim b_{n}$ means that $a_{n}=O(b_{n})$ and
$b_{n}=O(a_{n}).$

Now let us consider $h\left(  \mu\right)  $ on$\ \Gamma_{j}\left(
r_{n}\right)  $, $j=1,2$. Since $\mu_{n,2}^{1}$ is the root of $h\left(
\mu\right)  $ the Taylor expansion of $h\left(  \mu\right)  $ about $\mu
_{n,2}^{1}$ is
\begin{equation}
h\left(  \mu\right)  =h^{\prime}\left(  \mu_{n,2}^{1}\right)  \left(  \mu
-\mu_{n,2}^{1}\right)  +\frac{h^{\prime\prime}\left(  \mu_{n,2}^{1}\right)
}{2!}\left(  \mu-\mu_{n,2}^{1}\right)  ^{2}+\cdots.
\end{equation}
By (41), we have
\[
h^{\prime}\left(  \mu\right)  =\left(  \beta_{1}-1\right)  ie^{i\mu}+\left(
\beta_{2}-\frac{\beta_{1}+1}{2}c_{\mu}\right)  \frac{ie^{i\mu}}{i\mu}+O\left(
\dfrac{1}{n^{2}}\right)  \sim1
\]
for $\mu\in\Gamma_{j}\left(  r_{n}\right)  ,$ $j=1,2$. Clearly $h^{\left(
k\right)  }\left(  \mu\right)  \sim1$ for $k>1$ and $\mu\in\Gamma_{j}\left(
r_{n}\right)  $. On the other hand, $\left(  \mu-\mu_{n,2}^{1}\right)
\sim\dfrac{1}{n}$ for $\mu\in\ \Gamma_{1}\left(  r_{n}\right)  $ and $\left(
\mu-\mu_{n,2}^{1}\right)  \sim r_{n}$ for $\mu\in\ \Gamma_{2}\left(
r_{n}\right)  $. Therefore using (44) we obtain
\[
h\left(  \mu\right)  \sim\frac{1}{n},\text{ }\forall\mu\in\ \Gamma_{1}\left(
r_{n}\right)  ,
\]%
\[
h\left(  \mu\right)  \sim r_{n},\text{ }\forall\mu\in\ \Gamma_{2}\left(
r_{n}\right)  .
\]

These formulas with (40), (42) and (43) imply that
\begin{equation}
\Delta_{1}\left(  \mu\right)  \sim r_{n},\text{ }\forall\mu\in\ \Gamma
_{j}\left(  r_{n}\right)
\end{equation}
for $j=1,2.$ Thus by (38), (39), (45) and Rouche's theorem, each of the disks
enclosed by the circles $\Gamma_{1}\left(  r_{n}\right)  $ and $\Gamma
_{2}\left(  r_{n}\right)  $ contains an eigenvalue which proves (34) and (35).

Since the distance between the centres of the circles $\Gamma_{1}\left(
r_{n}\right)  $ and $\Gamma_{2}\left(  r_{n}\right)  $ is of order $\dfrac
{1}{n}$, but $r_{n}=o\left(  \dfrac{1}{n}\right)  $, the eigenvalues inside
the circles $\Gamma_{1}\left(  r_{n}\right)  $ and $\Gamma_{2}\left(
r_{n}\right)  $ are different, that is, they are simple.

Now let us prove (36). Since the equation
\[
-y^{\prime\prime}+q(x)y=\mu^{2}y
\]
has the fundamental solutions of the form
\[
y_{1}(x,\mu)=e^{i\mu x}+O\left(  \dfrac{1}{\mu}\right)  ,\text{ }y_{2}%
(x,\mu)=e^{-i\mu x}+O\left(  \dfrac{1}{\mu}\right)
\]
(see p. 52 of [18]) the eigenfunctions of $T_{1}^{1}\left(  q\right)  $ are
\begin{align*}
y_{n,j}\left(  x\right)   &  =%
\begin{vmatrix}
e^{i\mu_{n,j}x}+O\left(  \dfrac{1}{\mu_{n,j}}\right)  & e^{-i\mu_{n,j}%
x}+O\left(  \dfrac{1}{\mu_{n,j}}\right) \\
i\mu_{n,j}\left(  1+\beta_{1}e^{i\mu_{n,j}}\right)  +\beta_{2}e^{i\mu_{n,j}%
}+O\left(  \dfrac{1}{\mu_{n,j}}\right)  & -i\mu_{n,j}\left(  1+\beta
_{1}e^{-i\mu_{n,j}}\right)  +\beta_{2}e^{-i\mu_{n,j}}+O\left(  \dfrac{1}%
{\mu_{n,j}}\right)
\end{vmatrix}
\\
&  =\left[  e^{i\mu_{n,j}x}+O\left(  \dfrac{1}{\mu_{n,j}}\right)  \right]
\left[  -i\mu_{n,j}\left(  1+\beta_{1}e^{-i\mu_{n,j}}\right)  +\beta
_{2}e^{-i\mu_{n,1j}}+O\left(  \dfrac{1}{\mu_{n,j}}\right)  \right]  -\\
&  -\left[  e^{-i\mu_{n,j}x}+O\left(  \dfrac{1}{\mu_{n,j}}\right)  \right]
\left[  i\mu_{n,j}\left(  1+\beta_{1}e^{i\mu_{n,1}}\right)  +\beta_{2}%
e^{i\mu_{n,j}}+O\left(  \dfrac{1}{\mu_{n,j}}\right)  \right]  .
\end{align*}
This with the formula
\[
\mu_{n,j}=2\pi n+O\left(  \frac{1}{n}\right)  ,
\]
for $j=1,2$ (see (34) and (35)), implies (36).

$\left(  b\right)  $ It is clear that if (12) holds for the subsequence
$\left\{  n_{k}\right\}  $ then (36) holds for $\left\{  n_{k}\right\}  $ too.
Therefore the angle between the eigenfunctions $\varphi_{n_{k},1}\left(
x\right)  $ and $\varphi_{n_{k},2}\left(  x\right)  $ corresponding to
$\mu_{n_{k},1}\left(  q\right)  $ and $\mu_{n_{k},2}\left(  q\right)  $\ tends
to zero. Hence the system of the root functions of $T_{1}^{1}(q)$\ does not
form a Riesz basis (see [20]). Note that $\left(  b\right)  $ follows also
from $\left(  a\right)  $ and Theorem 2 of [12, 13].
\end{proof}

Let $q$ be an absolutely continuous function. Then using the integration by
parts formula for $s_{\mu}$ and $c_{\mu}$ defined in (24) we obtain
\[
s_{\mu}=\frac{1}{2\mu}\left[  q\left(  0\right)  -q\left(  1\right)
\cos\left(  2\mu\right)  \right]  +o(\frac{1}{\mu})
\]
and
\[
c_{\mu}=\frac{1}{2\mu}q\left(  1\right)  \sin\left(  2\mu\right)  +o\left(
\frac{1}{\mu}\right)  .
\]
If $\mu\in U(2\pi n),$ where $U(2\pi n)$ is defined in the proof of Lemma 1 ,
then%
\[
\cos\mu=1+O\left(  \frac{1}{\mu}\right)  \&\sin\mu=O\left(  \frac{1}{\mu
}\right)
\]
Therefore we have
\[
s_{\mu}=\frac{1}{2\mu}\left[  q\left(  0\right)  -q\left(  1\right)  \right]
+o(\frac{1}{\mu}),\text{ }c_{\mu}=o\left(  \frac{1}{\mu}\right)
\]
and hence by (25)%
\begin{gather}
\Delta\left(  \mu\right)  =\Delta_{0}\left(  \mu\right)  +i\left(  \beta
_{1}+1\right)  s_{\mu}\cos\mu+o\left(  \frac{1}{\mu}\right) \nonumber\\
=\Delta_{0}\left(  \mu\right)  +\frac{a}{\mu}+o\left(  \frac{1}{\mu}\right)
\end{gather}
where
\[
a=\frac{i\left(  \beta_{1}+1\right)  }{2}\left[  q\left(  0\right)  -q\left(
1\right)  \right]  .
\]
Now we are ready to state the second main result of this paper.

\begin{theorem}
Let $q$ be an absolutely continuous function and (13) for $\sigma=1$ hold. Then

$\left(  a\right)  $ the large eigenvalues of $T_{1}^{1}(q)$ are simple and
the square roots (with nonnegative real part) of these eigenvalues consist of
two sequences $\{\mu_{n,1}(q)\}$ and $\{\mu_{n,2}(q)\}$ satisfying
\begin{equation}
\mu_{n,1}(q)=2\pi n+\frac{2\beta_{2}-i\sqrt{D}}{4\left(  \beta_{1}-1\right)
\pi n}+o\left(  \frac{1}{n}\right)  \text{,}%
\end{equation}%
\begin{equation}
\mu_{n,2}(q)=2\pi n+\frac{2\beta_{2}+i\sqrt{D}}{4\left(  \beta_{1}-1\right)
\pi n}+o\left(  \frac{1}{n}\right)  .
\end{equation}
where $D=2\left(  1-\beta_{1}^{2}\right)  \left[  q\left(  0\right)  -q\left(
1\right)  \right]  -\left(  2\beta_{2}\right)  ^{2}$

$\left(  b\right)  $ the system of the root functions of $T_{1}^{1}(q)$ does
not form a Riesz basis.
\end{theorem}

\begin{proof}
$\left(  a\right)  $ By (46) $\mu_{n,j}(q)$ is a root of the equation%
\[
\mu\Delta_{0}\left(  \mu\right)  +a+o\left(  1\right)  =0.
\]
Using (17) in this equation we get
\begin{equation}
\mu\left(  1-e^{-i\mu}\right)  \left[  i\mu\left(  \beta_{1}-1\right)  \left(
e^{i\mu}-1\right)  +\beta_{2}\left(  e^{i\mu}+1\right)  \right]  +a+o\left(
1\right)  =0.
\end{equation}
By the Taylor expansions of $e^{-i\mu}$ and $e^{i\mu}$ at $2\pi n$ we have
\begin{align*}
e^{-i\mu}  &  =1-i\left(  \mu-2\pi n\right)  +O\left(  \frac{1}{n^{2}}\right)
,\\
e^{i\mu}  &  =1+i\left(  \mu-2\pi n\right)  +O\left(  \frac{1}{n^{2}}\right)
\end{align*}
for $\mu\in U(2\pi n).$ Therefore (49) can be written in the form
\begin{equation}
i\mu\left(  \mu-2\pi n\right)  \left[  -\mu\left(  \beta_{1}-1\right)  \left(
\mu-2\pi n\right)  +2\beta_{2}+O\left(  \frac{1}{\mu}\right)  \right]
+a+o\left(  1\right)  =0.
\end{equation}
To prove the formulas (47) and (48) we consider the equation (50). In (50)
substituting $x=\mu\left(  \mu-2\pi n\right)  $ and taking into account that
$x=O(1)$ for $\mu\in U(2\pi n)$ we get
\begin{equation}
-i\left(  \beta_{1}-1\right)  x^{2}+2i\beta_{2}x+a+o\left(  1\right)  =0.
\end{equation}
To solve (51) we compare the roots of the functions
\begin{equation}
f_{1}\left(  \mu\right)  =-i\left(  \beta_{1}-1\right)  x^{2}+2i\beta_{2}x+a
\end{equation}
and%
\begin{equation}
f_{2}\left(  \mu\right)  =-i\left(  \beta_{1}-1\right)  x^{2}+2i\beta
_{2}x+a+\alpha_{n}%
\end{equation}
on the set $U(2\pi n),$ where $\alpha_{n}=o\left(  1\right)  $. The roots of
$f_{1}\left(  \mu\right)  $ are
\begin{equation}
x_{1,2}=\frac{-2i\beta_{2}\pm\sqrt{D}}{-2i\left(  \beta_{1}-1\right)  }%
\end{equation}
where
\begin{equation}
D=\left(  2i\beta_{2}\right)  ^{2}+4i\left(  \beta_{1}-1\right)  a=\left(
2i\beta_{2}\right)  ^{2}-2\left(  \beta_{1}^{2}-1\right)  \left[  q\left(
0\right)  -q\left(  1\right)  \right]  \neq0.
\end{equation}
by the assumption (13) for $\sigma=1$. Therefore we have two different
solutions $x_{1}$ and $x_{2}$.

On the other hand the solutions of the equations $\mu\left(  \mu-2\pi
n\right)  =x_{1}$ and $\mu\left(  \mu-2\pi n\right)  =x_{2}$ with respect to
$\mu$\ are
\[
\mu_{11}=O\left(  \frac{1}{n}\right)  ,\text{ }\mu_{12}=2\pi n+\frac{x_{1}%
}{2\pi n}+O\left(  \frac{1}{n^{2}}\right)
\]
and
\[
\mu_{21}=O\left(  \frac{1}{n}\right)  ,\text{ }\mu_{22}=2\pi n+\frac{x_{2}%
}{2\pi n}+O\left(  \frac{1}{n^{2}}\right)
\]
respectively. Since $x_{1}-x_{2}\sim1$ (see (54) and (55)), we have
\begin{equation}
\mu_{12}-\mu_{21}\sim n,\text{ }\mu_{12}-\mu_{22}\sim\frac{1}{n},\text{ }%
\mu_{12}-\mu_{11}\sim n.
\end{equation}
Now consider the roots of $f_{2}\left(  \mu\right)  $ by using Rouche's
theorem on
\begin{equation}
\gamma_{j}\left(  r_{n}\right)  =\left\{  \mu:\text{ }\left\vert \mu-\mu
_{j2}\right\vert =r_{n}\right\}  ,
\end{equation}
for $j=1,2$, where $r_{n}$ is chosen so that
\begin{equation}
r_{n}=o\left(  \dfrac{1}{n}\right)  \text{ \& }\alpha_{n}=o\left(
nr_{n}\right)  .
\end{equation}
By (52), (53) and (58)
\[
\left\vert f_{1}\left(  \mu\right)  -f_{2}\left(  \mu\right)  \right\vert
=\alpha_{n}=o\left(  1\right)
\]
on $\gamma_{1}\left(  r_{n}\right)  \cap\gamma_{2}\left(  r_{n}\right)  $.
Since the roots of $f_{1}\left(  \mu\right)  $ are $\mu_{ij}$ for $i,j=1,2$,
we have
\begin{equation}
f_{1}\left(  \mu\right)  =A\left(  \mu-\mu_{11}\right)  \left(  \mu-\mu
_{12}\right)  \left(  \mu-\mu_{21}\right)  \left(  \mu-\mu_{22}\right)
\end{equation}
where $A$ is a constant. One can easily verify by using (56) and (59) that
\[
f^{\prime}\left(  \mu_{12}\right)  =A\left(  \mu_{12}-\mu_{11}\right)  \left(
\mu_{12}-\mu_{21}\right)  \left(  \mu_{12}-\mu_{22}\right)  \sim n
\]
Since $f\left(  \mu\right)  $ is a polynomial of order $4$ we have
\[
f^{\prime\prime}\left(  \mu_{12}\right)  =O(n^{2}),\text{ }f^{\prime
\prime\prime}\left(  \mu_{12}\right)  =O(n),\text{ }f^{\left(  4\right)
}\left(  \mu_{12}\right)  =O(1),\text{ }f^{\left(  5\right)  }\left(  \mu
_{12}\right)  =0.
\]
Therefore using the Taylor series
\[
f_{1}\left(  \mu\right)  =f_{1}^{\prime}\left(  \mu_{12}\right)  \left(
\mu-\mu_{12}\right)  +\cdots.
\]
of $f_{1}\left(  \mu\right)  $\ about $\mu_{12}$ for $\mu\in\gamma_{1}\left(
r_{n}\right)  $ and taking into account that $\left(  \mu-\mu_{12}\right)
\sim r_{n}$ we obtain
\[
\left\vert f_{1}\left(  \mu\right)  \right\vert \sim nr_{n}.
\]
On the other hand by (58) we have
\[
\left\vert f_{1}\left(  \mu\right)  -f_{2}\left(  \mu\right)  \right\vert
=\alpha_{n}=o\left(  nr_{n}\right)
\]
for $\mu\in\gamma_{1}\left(  r_{n}\right)  $. Therefore
\begin{equation}
\left\vert f_{1}\left(  \mu\right)  -f_{2}\left(  \mu\right)  \right\vert
<\left\vert f_{1}\left(  \mu\right)  \right\vert
\end{equation}
on $\gamma_{1}\left(  r_{n}\right)  $ In the same way we prove that (60) holds
on $\gamma_{2}\left(  r_{n}\right)  $ too. Hence inside of each of the circles
$\gamma_{1}\left(  r_{n}\right)  $ and $\gamma_{2}\left(  r_{n}\right)  $,
there is one root of (49) denoted by $\mu_{n,1}\left(  q\right)  $\ and
$\mu_{n,2}\left(  q\right)  $ respectively. Since $r_{n}=o\left(  \frac{1}%
{n}\right)  ,$ $\mu_{n,1}\left(  q\right)  $\ and $\mu_{n,2}\left(  q\right)
$ satisfy the formulas (47) and (48). To complete the proof of $\left(
a\right)  $ it is enough to note that disks enclosed by the circles
$\gamma_{1}\left(  r_{n}\right)  $ and $\gamma_{2}\left(  r_{n}\right)  $ have
no common points and there are only two roots of (46) in the neighborhood of
$2\pi n$. Thus $\left(  a\right)  $ is proved.

$\left(  b\right)  $ The proof of $\left(  b\right)  $ is the same as the
proof of Theorem 1$\left(  b\right)  $.
\end{proof}

Now consider $T_{1}^{0}\left(  q\right)  $. In this case the characteristic
determinant of $T_{1}^{0}\left(  0\right)  $ is
\[
\Delta_{0}^{0}\left(  \mu\right)  =\left(  1+e^{i\mu}\right)  \left(
i\mu+\beta_{1}i\mu e^{-i\mu}-\beta_{2}e^{-i\mu}\right)  +\left(  i\mu
+\beta_{1}i\mu e^{i\mu}+\beta_{2}e^{i\mu}\right)  \left(  1+e^{-i\mu}\right)
=0.
\]
After simplifying this equation, we have%
\[
\Delta_{0}^{0}\left(  \mu\right)  =\left(  1+e^{-i\mu}\right)  \left[
i\mu\left(  \beta_{1}+1\right)  \left(  e^{i\mu}+1\right)  +\beta_{2}\left(
e^{i\mu}-1\right)  \right]  =0.
\]
The roots of this equation has the form%
\[
\left(  2n+1\right)  \pi,\text{ }\left(  2n+1\right)  \pi+\frac{2\beta_{2}%
}{\beta_{1}+1}\frac{1}{\left(  2n+1\right)  \pi}+O\left(  \frac{1}{n^{2}%
}\right)  .
\]

The characteristic determinant of $T_{1}^{0}\left(  q\right)  $ can be written
in the forms%

\[
\Delta^{0}\left(  \mu\right)  =\Delta_{0}^{0}\left(  \mu\right)
+\frac{1-\beta_{1}}{2}e^{-i\mu}\left\{  c_{\mu}\left(  e^{2i\mu}-1\right)
-is_{\mu}\left(  e^{2i\mu}+1\right)  \right\}  +o\left(  \frac{1}{\mu}\right)
\]
and
\[
\Delta^{0}\left(  \mu\right)  =\Delta_{1}^{0}\left(  \mu\right)  +i\left(
\beta_{1}-1\right)  s_{\mu}\cos\mu+o\left(  \frac{1}{\mu}\right)  ,
\]
where
\[
\Delta_{1}^{0}\left(  \mu\right)  =\left(  1+e^{-i\mu}\right)  \left[
i\mu\left(  \beta_{1}+1\right)  \left(  e^{i\mu}+1\right)  +\left(  \beta
_{2}+\frac{1-\beta_{1}}{2}c_{\mu}\right)  \left(  e^{i\mu}-1\right)  \right]
.
\]

The investigation of $T_{1}^{0}\left(  q\right)  $ is similar to the
investigation of $T_{1}^{1}(q).$ The difference is that, here we consider the
functions and equations in $O\left(  \dfrac{1}{n}\right)  $-neighborhood of
$\left(  2n+1\right)  \pi$ (we denote it by $U(\left(  2n+1\right)  \pi)$)
instead of $U(2\pi n),$ since the eigenvalues of $T_{1}^{0}(0)$ lie in
$U(\left(  2n+1\right)  \pi)$ while the eigenvalues of $T_{1}^{1}(0)$ lie in
$U(2\pi n).$ Now instead of $\Delta_{0},$ $\Delta_{1},$ $\Delta$ using the
functions $\Delta_{0}^{0},$ $\Delta_{1}^{0},$ $\Delta^{0}$ and repeating the
proof of Theorem 1 we obtain:

\begin{theorem}
$\left(  a\right)  $ If (12a) holds, then the large eigenvalues of $T_{1}%
^{0}(q)$ are simple and the square roots (with nonnegative real part) of these
eigenvalues consist of two sequences $\{\mu_{n,1}^{0}\}$ and $\{\mu_{n,2}%
^{0}\}$ satisfying
\[
\mu_{n,1}^{0}=\text{ }\left(  2n+1\right)  \pi+o\left(  \frac{1}{n}\right)
\text{,}%
\]%
\[
\mu_{n,2}^{0}=\left(  2n+1\right)  \pi+\frac{2\beta_{2}}{\beta_{1}+1}\frac
{1}{\left(  2n+1\right)  \pi}+o\left(  \frac{1}{n}\right)  .
\]
Moreover the normalized eigenfunctions $\varphi_{n,1}^{0}\left(  x\right)  $
and $\varphi_{n,2}^{0}\left(  x\right)  $ corresponding to the eigenvalues
$\left(  \mu_{n,1}^{0}\right)  ^{2}$ and $\left(  \mu_{n,2}^{0}\right)  ^{2}%
$satisfy the same asymptotic formula
\[
\varphi_{n,j}^{0}\left(  x\right)  =\sqrt{2}\cos\left(  2n+1\right)  \pi
x+O\left(  \frac{1}{n}\right)  .
\]
for $j=1,2.$

$\left(  b\right)  $ If there exists a subsequence $\left\{  n_{k}\right\}  $
such that (12a) holds whenever $n$ is replaced by $n_{k}$, then the system of
the root functions of $T_{1}^{0}(q)$ does not form a Riesz basis.
\end{theorem}

Now we investigate $T_{1}^{0}(q),$ when $q$ is an absolutely continuous
function. The analogous formula to (46) is%

\begin{equation}
\Delta^{0}\left(  \mu\right)  =\Delta_{0}^{0}\left(  \mu\right)  +\frac{b}%
{\mu}+o\left(  \frac{1}{\mu}\right)  =0,
\end{equation}
where
\[
b=\frac{i\left(  1-\beta_{1}\right)  }{2}\left[  q\left(  0\right)  +q\left(
1\right)  \right]  .
\]

Instead of (46) using (61) and repeating the proof of Theorem 2, we obtain:

\begin{theorem}
Let $q$ be an absolutely continuous function and (13) for $\sigma=0$ hold.

$\left(  a\right)  $ The large eigenvalues of $T_{1}^{0}(q)$ are simple and
the square roots (with nonnegative real part) of these eigenvalues consist of
two sequences $\{\mu_{n,1}^{0}\}$ and $\{\mu_{n,2}^{0}\}$ satisfying
\[
\mu_{n,1}^{0}=\left(  2n+1\right)  \pi+\frac{2\beta_{2}-i\sqrt{D_{2}}%
}{2\left(  \beta_{1}+1\right)  \left(  2n+1\right)  \pi}+o\left(  \frac{1}%
{n}\right)  ,
\]%
\[
\mu_{n,2}^{0}=\left(  2n+1\right)  \pi+\frac{2\beta_{2}+i\sqrt{D_{2}}%
}{2\left(  \beta_{1}+1\right)  \left(  2n+1\right)  \pi}+o\left(  \frac{1}%
{n}\right)  ,
\]
where $D_{2}=2\left(  1-\beta_{1}^{2}\right)  \left[  q\left(  0\right)
+q\left(  1\right)  \right]  -\left(  2\beta_{2}\right)  ^{2}.$

$\left(  b\right)  $ The system of the root functions of $T_{1}^{0}(q)$ does
not form a Riesz basis.
\end{theorem}

Now we consider $T_{2}^{1}\left(  q\right)  $. In this case the characteristic
determinant of $T_{2}^{1}\left(  0\right)  $ is
\begin{equation}
D_{0}^{1}\left(  \mu\right)  =\left(  1-e^{i\mu}\right)  \left(  \beta_{3}%
i\mu+i\mu e^{-i\mu}-\beta_{4}e^{-i\mu}\right)  +\left(  \beta_{3}i\mu+i\mu
e^{i\mu}+\beta_{4}e^{i\mu}\right)  \left(  1-e^{-i\mu}\right)  =0.\nonumber
\end{equation}
After simplifying this equation, we have%
\[
D_{0}^{1}\left(  \mu\right)  =\left(  1-e^{-i\mu}\right)  \left[  i\mu\left(
1-\beta_{3}\right)  \left(  e^{i\mu}-1\right)  +\beta_{4}\left(  e^{i\mu
}+1\right)  \right]  =0.
\]
The roots of this equation has the form
\[
2\pi n,\text{ }2\pi n+\frac{\beta_{4}}{1-\beta_{3}}\frac{1}{\pi n}+O\left(
\frac{1}{n^{2}}\right)  .
\]
The characteristic determinant of $T_{2}^{1}\left(  q\right)  $ can be written
in the forms
\[
D^{1}\left(  \mu\right)  =D_{0}^{1}\left(  \mu\right)  -\frac{\beta_{3}+1}%
{2}e^{-i\mu}\left\{  c_{\mu}\left(  e^{2i\mu}-1\right)  -is_{\mu}\left(
e^{2i\mu}+1\right)  \right\}  +o\left(  \frac{1}{\mu}\right)
\]
and
\[
D^{1}\left(  \mu\right)  =D_{1}^{1}\left(  \mu\right)  +i\left(  \beta
_{3}+1\right)  s_{\mu}\cos\mu+o\left(  \frac{1}{\mu}\right)  ,
\]
where
\[
D_{1}^{1}\left(  \mu\right)  =\left(  1-e^{-i\mu}\right)  \left[  i\mu\left(
1-\beta_{3}\right)  \left(  e^{i\mu}-1\right)  +\left(  \beta_{4}-\frac
{\beta_{3}+1}{2}c_{\mu}\right)  \left(  e^{i\mu}+1\right)  \right]  .
\]

Instead of $\Delta_{0},$ $\Delta_{1},$ $\Delta$ using the functions $D_{0}%
^{1},$ $D_{1}^{1},$ $D^{1}$ and repeating the proof of Theorem 1 we obtain:

\begin{theorem}
$\left(  a\right)  $ If (12) holds, then the large eigenvalues of $T_{2}%
^{1}(q)$ are simple and the square roots (with nonnegative real part) of these
eigenvalues consist of two sequences $\{\rho_{n,1}\}$ and $\{\rho_{n,2}\}$
satisfying
\[
\rho_{n,1}=2\pi n+o\left(  \frac{1}{n}\right)  \text{,}%
\]%
\[
\rho_{n,2}=2\pi n+\frac{\beta_{4}}{1-\beta_{3}}\frac{1}{\pi n}+o\left(
\frac{1}{n}\right)  .
\]
Moreover the normalized eigenfunctions $\phi_{n,1}\left(  x\right)  $ and
$\phi_{n,2}\left(  x\right)  $ corresponding to the eigenvalues $\left(
\rho_{n,1}\right)  ^{2}$ and $\left(  \rho_{n,2}\right)  ^{2}$satisfy the same
asymptotic formula
\[
\phi_{n,j}\left(  x\right)  =\sqrt{2}\cos2\pi nx+O\left(  \frac{1}{n}\right)
\]
for $j=1,2$

$\left(  b\right)  $ If there exists a subsequence $\left\{  n_{k}\right\}  $
such that (12) holds whenever $n$ is replaced by $n_{k}$, then the system of
the root functions of $T_{2}^{1}(q)$ does not form a Riesz basis.
\end{theorem}

Let $q$ be an absolutely continuous function. Then analogous formula to (46) is%

\begin{equation}
D^{1}\left(  \mu\right)  =D_{0}^{1}\left(  \mu\right)  +\frac{c}{\mu}+o\left(
\frac{1}{\mu}\right)  =0,
\end{equation}
where
\[
c=\frac{i\left(  \beta_{3}+1\right)  }{2}\left[  q\left(  0\right)  -q\left(
1\right)  \right]  .
\]

Now instead of (46) using (62) and repeating the proof of Theorem 2, we obtain:

\begin{theorem}
Let $q$ be an absolutely continuous function and (14) for $\sigma=1$ hold. Then

$\left(  a\right)  $ the large eigenvalues of $T_{2}^{1}(q)$ are simple and
the square roots (with nonnegative real part) of these eigenvalues consist of
two sequences $\{\rho_{n,1}\}$ and $\{\rho_{n,2}\}$ satisfying
\[
\rho_{n,1}=2\pi n+\frac{-2\beta_{4}-i\sqrt{D_{3}}}{4\left(  \beta
_{3}-1\right)  \pi n}+o\left(  \frac{1}{n}\right)  ,
\]%
\[
\rho_{n,2}=2\pi n+\frac{-2\beta_{4}+i\sqrt{D_{3}}}{4\left(  \beta
_{3}-1\right)  \pi n}+o\left(  \frac{1}{n}\right)  ,
\]
where $D_{3}=2\left(  \beta_{3}^{2}-1\right)  \left[  q\left(  0\right)
-q\left(  1\right)  \right]  -\left(  2\beta_{4}\right)  ^{2}.$

$\left(  b\right)  $ the system of the root functions of $T_{2}^{1}(q)$ does
not form a Riesz basis.
\end{theorem}

Finally, we consider $T_{2}^{0}\left(  q\right)  $. In this case the
characteristic determinant of $T_{2}^{0}\left(  0\right)  $ is
\begin{equation}
D_{0}^{0}\left(  \mu\right)  =\left(  1+e^{i\mu}\right)  \left(  \beta_{3}%
i\mu+i\mu e^{-i\mu}-\beta_{4}e^{-i\mu}\right)  +\left(  \beta_{3}i\mu+i\mu
e^{i\mu}+\beta_{4}e^{i\mu}\right)  \left(  1+e^{-i\mu}\right)  =0.\nonumber
\end{equation}
After simplifying this equation, we have%
\[
D_{0}^{0}\left(  \mu\right)  =\left(  1+e^{-i\mu}\right)  \left[  i\mu\left(
1+\beta_{3}\right)  \left(  e^{i\mu}+1\right)  +\beta_{4}\left(  e^{i\mu
}-1\right)  \right]  =0.
\]
The roots of this equation has the form
\[
\left(  2n+1\right)  \pi,\text{ }\left(  2n+1\right)  \pi+\frac{2\beta_{4}%
}{\beta_{3}+1}\frac{1}{\left(  2n+1\right)  \pi}+O\left(  \frac{1}{n^{2}%
}\right)  .
\]

The characteristic determinant of $T_{2}^{1}\left(  q\right)  $ can be written
in the forms%

\[
D^{0}\left(  \mu\right)  =D_{0}^{0}\left(  \mu\right)  +\frac{\beta_{3}-1}%
{2}e^{-i\mu}\left\{  c_{\mu}\left(  e^{2i\mu}-1\right)  -is_{\mu}\left(
e^{2i\mu}+1\right)  \right\}  +o\left(  \frac{1}{\mu}\right)
\]
and
\[
D^{0}\left(  \mu\right)  =D_{1}^{0}\left(  \mu\right)  +i\left(  1-\beta
_{3}\right)  s_{\mu}\cos\mu+o\left(  \frac{1}{\mu}\right)  ,
\]
where
\[
D_{1}^{0}\left(  \mu\right)  =\left(  1+e^{-i\mu}\right)  \left[  i\mu\left(
1+\beta_{3}\right)  \left(  e^{i\mu}+1\right)  +\left(  \beta_{4}+\frac
{\beta_{3}-1}{2}c_{\mu}\right)  \left(  e^{i\mu}-1\right)  \right]  .
\]

Instead of $\Delta_{0},$ $\Delta_{1},$ $\Delta$ using the functions $D_{0}%
^{0},D_{1}^{0},D^{0}$ and repeating the proof of Theorem 1 we obtain:

\begin{theorem}
$\left(  a\right)  $ If (12a) holds, then the large eigenvalues of $T_{2}%
^{0}(q)$ are simple and the square roots (with nonnegative real part) of these
eigenvalues consist of two sequences $\{\rho_{n,1}^{0}\}$ and $\{\rho
_{n,2}^{0}\}$ satisfying
\[
\rho_{n,1}^{0}=\left(  2n+1\right)  \pi+o\left(  \frac{1}{n}\right)  \text{,}%
\]%
\[
\rho_{n,2}^{0}=\left(  2n+1\right)  \pi+\frac{2\beta_{4}}{\beta_{3}+1}\frac
{1}{\left(  2n+1\right)  \pi}+o\left(  \frac{1}{n}\right)  .
\]
Moreover the normalized eigenfunctions $\phi_{n,1}^{0}\left(  x\right)  $ and
$\phi_{n,2}^{0}\left(  x\right)  $ corresponding to the eigenvalues $\left(
\rho_{n,1}^{0}\right)  ^{2}$ and $\left(  \rho_{n,2}^{0}\right)  ^{2}$satisfy
the same asymptotic formula
\[
\phi_{n,j}^{0}\left(  x\right)  =\sqrt{2}\cos\left(  2n+1\right)  \pi
x+O\left(  \frac{1}{n}\right)
\]
for $j=1,2.$

$\left(  b\right)  $ If there exists a subsequence $\left\{  n_{k}\right\}  $
such that (12a) holds whenever $n$ is replaced by $n_{k}$, then the system of
the root functions of $T_{2}^{0}(q)$ does not form a Riesz basis.
\end{theorem}

Let $q$ be an absolutely continuous function. Then analogous formula to (46)
is
\begin{equation}
D^{0}\left(  \mu\right)  =D_{0}^{0}\left(  \mu\right)  +\frac{d}{\mu}+o\left(
\frac{1}{\mu}\right)  =0,
\end{equation}
where
\[
d=\frac{i\left(  \beta_{3}-1\right)  }{2}\left[  q\left(  0\right)  +q\left(
1\right)  \right]  .
\]

Now instead of (46) using (63) and repeating the proof of Theorem 2, we obtain:

\begin{theorem}
Let $q$ be an absolutely continuous function and (14) for $\sigma=0$ hold. Then

$\left(  a\right)  $ the large eigenvalues of $T_{2}^{0}(q)$ are simple and
the square roots (with nonnegative real part) of these eigenvalues consist of
two sequences $\{\rho_{n,1}^{0}\}$ and $\{\rho_{n,2}^{0}\}$ satisfying
\[
\rho_{n,1}^{0}=\left(  2n+1\right)  \pi+\frac{2\beta_{4}-i\sqrt{D_{4}}%
}{2\left(  \beta_{3}+1\right)  \left(  2n+1\right)  \pi}+o\left(  \frac{1}%
{n}\right)  ,
\]%
\[
\rho_{n,2}^{0}=\left(  2n+1\right)  \pi+\frac{2\beta_{4}+i\sqrt{D_{4}}%
}{2\left(  \beta_{3}+1\right)  \left(  2n+1\right)  \pi}+o\left(  \frac{1}%
{n}\right)  ,
\]
where $D_{4}=2\left(  \beta_{3}^{2}-1\right)  \left[  q\left(  0\right)
+q\left(  1\right)  \right]  -\left(  2\beta_{4}\right)  ^{2}.$

$\left(  b\right)  $ the system of the root functions of $T_{2}^{0}(q)$ does
not form a Riesz basis.
\end{theorem}

\end{document}